\documentclass[11pt]{article}

\usepackage[margin=2.5cm,left=3cm,right=3cm]{geometry}

\usepackage{amsmath,amsthm,amsfonts,amssymb,fancyhdr,graphics}
\usepackage{float}
\usepackage[pdftex]{graphicx} 
\usepackage[utf8]{inputenc}
\usepackage[english, activeacute]{babel} 
\usepackage[pdftex]{graphicx} 
\DeclareGraphicsExtensions{.png,.pdf,.jpg} 
\usepackage{float}
\usepackage{appendix}
\usepackage{listings}
\usepackage[usenames,dvipsnames,svgnames,table]{xcolor}
\usepackage[all]{xy}
\usepackage{mathrsfs}
\usepackage{tikz} 
\usepackage{hyperref} 
\hypersetup{
    colorlinks=true
    } 
\usepackage[T1]{fontenc}
\usepackage{lmodern}	
\usepackage{soul}

\newcommand{\N}{\mathbb{N}}

\newcommand{\Q}{\mathbb{Q}}
\newcommand{\R}{\mathbb{R}}
\newcommand{\C}{\mathbb{C}}
\newcommand{\A}{\mathbb{A}}

\newcommand{\PP}{\mathbb{P}}

\DeclareMathOperator{\Supp}{Supp}

\DeclareMathOperator{\mult}{mult}

\DeclareMathOperator{\ord}{ord}
\DeclareMathOperator{\vol}{vol}

\DeclareMathOperator{\wt}{wt}

\DeclareMathOperator{\Aut}{Aut}

\DeclareMathOperator{\Diff}{Diff}

\theoremstyle{plain}

\newtheorem{lemma}{Lemma}[section]

\newtheorem{proposition}[lemma]{Proposition}
\newtheorem{theorem}[lemma]{Theorem}
\newtheorem{corollary}[lemma]{Corollary}

\theoremstyle{definition}

\newtheorem{definition}[lemma]{Definition}
\newtheorem{remark}[lemma]{Remark}

\newtheorem*{definition*}{Definition}
\newtheorem*{remark*}{Remark}
\newtheorem*{remarks*}{Remarks}
\newtheorem*{example*}{Example}
\newtheorem*{examples*}{Examples}
\newtheorem*{conjecture*}{Conjecture}
\newtheorem*{conjectures*}{Conjectures}
\newtheorem*{exercise*}{Exercise}
\newtheorem*{exercises*}{Exercises}
\newtheorem*{problem*}{Problem}
\newtheorem*{problems*}{Problems}

\addto\captionsenglish{ }

\title{On local stability threshold of del Pezzo surfaces}
\author{Erroxe Etxabarri-Alberdi}
\date{}
\newcommand{\Addresses}{{
  \bigskip
  \footnotesize

  \textsc{School of Mathematical Sciences, University of Nottingham, Nottingham, NG7~2RD, United Kingdom}\par\nopagebreak
  \textit{E-mail address}: \texttt{Erroxe.EtxabarriAlberdi@nottingham.ac.uk}
}}
\begin{document}

\maketitle

\begin{abstract}
We complete the classification of local stability thresholds for smooth del Pezzo surfaces of degree~2. In particular, we show that this number is irrational if and only if {there is a unique (-1)-curve passing through the point where we are computing the local invariant.}
\end{abstract}

\section{Introduction}

K-stability is an algebraic condition that detects the existence of Kähler-Einstein metrics on Fano varieties \cite{Tia97,Don02,Tia15,LXZ22,CDS15}. It also provides the right framework to construct compact moduli for Fano varieties \cite{ABHLX20, BHLLX21, BLX22, BX19, Jia18, LWX21, LXZ22, Xu20, XZ20}.

\vskip 0.3cm
In this paper, we will focus on the study $\delta$, {mostly} in the local situation. \\
It was originally conjectured that $\delta(X)$ is rational for K-polystable Fano varieties, and was proven by Liu, Xu, and Zhuang in \cite{LXZ22}
{when} $\delta (X)<(n+1)/n$ where $n$ is the dimension of $X$. However, the local stability threshold, $\delta_p(X)$, is more mysterious. {Its rationality can be unpredictable}. Note that $\delta(X)$ is the infimum of $\delta_p(X)$ for all $p\in X$. The only known example where the local delta invariant is irrational is in cubic surfaces (see \cite[Lemma A.6]{AZ20}).
{The rationality of $\delta_p(X)$ is well-established for all points $p$ in the case of del Pezzo and weak del Pezzo surfaces with a degree equal to or greater than 4 (refer to \cite[\S 2]{acc2021}, \cite{Aka23, Den23})}.
It is expected that irrational $\delta_p$ exists only for certain points $p$ in del Pezzo surfaces of degrees 1, 2, and 3. In this paper, we settle this in degree 2 with the following theorem.


\begin{theorem}\label{maintheo}
Let $X\subset \PP (1,1,1,2)$ be a smooth del Pezzo surface of degree 2 and let $p_0=(x,y,z,w)\in X$ be a closed point with $w\neq 0$. Assume {that there is a unique (-1)-curve $L$ passing through the point $p_0$}. Then $\delta_{p_0}(X)=\frac{6}{71} (11 + 8 \sqrt{3})$ and it is computed by {taking the limit of a sequence of} weighted blow ups at $p_0$ with {$\wt(u)=a_m$ and $\wt(v)=b_m$ such that $a_m/b_m\to 2/\sqrt{3}$ when $m\to \infty $}, where $u$ and $v$ are the local coordinates such that $L=\{ u=0\}$.
\end{theorem}


\paragraph{Sketch of the proof.}

In order to prove that the local stability threshold at {$p_0$ (as in Theorem \ref{maintheo})} is irrational, we follow these steps:
First, we take $E_{a,b}$ to be the exceptional divisor of a certain weighted blowup at $p_0$ with weights $(a,b)$. The choice of $(a,b)$ becomes clear in \S \ref{pmaintheo}. Using that exceptional divisor we compute an upper bound for $\delta_{p_0}(X)$. The difficulty here is to compute the expected vanishing order of the anticanonical divisor of $X$ with respect to {$E_{a,b}$}, which requires a Zariski Decomposition and a careful choice of negative curves in the weighted blowup of $X$. This is the main bulk of the work in section \S \ref{Finding D}.\\
Then, we use techniques from \cite{AZ20} to find lower bounds for $\delta_{p_0}(X)$, which requires a delicate choice of a minimazer sequence of prime divisors $E_{a_m,b_m}$ over $X$. The choice of $a_m$, $b_m$ is made so that this bound is exactly the upper bound found earlier by considering $E_{a,b}$. In other words, we compute $\delta_{p_0}(X)$.

\paragraph{Acknowledgements:} I thank Hamid Abban for suggesting the problem to me and his constant guidance. Moreover, I want to thank Tiago Duarte Guerreiro, Luca Giovenzana, James Jones, Antonio Trusiani, Nivedita Viswanathan and Ziquan Zhuang for numerous discussions and the interest shown.

\section{Preliminary results and definitions}

\subsection{Singularities}

In this paper we work over $\C$ and all varieties are assumed to be normal and projective. 

\begin{definition}
	A \textit{logarithmic pair} (also called a \textit{log pair}) is a pair $(X,D)$ consisting of a normal variety $X$ and a boundary $\Q$-divisor $D$. A boundary divisor $D=\sum d_kD_k$ is a divisor with rational coefficients between $0$ and $1$, i.e. $0\leq d_k\leq 1$, where the $D_k$ are prime divisors. We call $K_X+D$ the \textit{log canonical divisor} of the pair $(X,D)$.
\end{definition}

\begin{definition}
	Let $(X,D)$ be a log pair. Assume that $K_X+D$ is $\Q$-Cartier. Let $f: V\to X$ be a birational morphism from a nonsingular variety $V$. Write the ramification formula:
	\[
	K_V=f^*(K_X+D)-f_*^{-1}(D)+\sum a(E;X,D)E,
	\]
	where the sum runs over $f$-exceptional divisors $E$. Recall that $a(E;X,D)$ is a rational number independent of the morphism $f$ and depends only on the discrete valuation which corresponds to $E$. This is called the \textit{discrepancy} of the pair $(X,D)$ at $E$. The \textit{log discrepancy} of the pair $(X,D)$ at $E$  is defined to be
	\[
	A_{X,D}(E)=a(E;X,D)+1.
 \]
\end{definition}

There is a more general definition of log discrepancy in \cite[Def. 2.5]{LXZ22}.

Since we are working with Fano surfaces, it is essential to recall the characterization for ampleness and other properties of divisors which you can find in different references such as \cite{Matsuki,Laz04}.

\begin{definition}
	Let $r>0$ and $a_1,...,a_n$ be integers and let $x_1,...,x_n$ be coordinates on $\A^n$. Suppose that $\mu_r$ acts on $\A^n$ via:
	\[
	x_i\mapsto \varepsilon^{a_i}x_i \text{ for all } i,
	\]
	where $\varepsilon$ is a fixed primitive $r$-th root of unity. A singurality $\mathcal{Q}\in X$ is a \textit{quotient singularity} of \textit{type} $\frac{1}{r}(a_1,...,a_n)$ if $(X,\mathcal{Q})$ is isomorphic to an analytic neighbourhood of $(\A^n, 0)/\mu_r$, where $\mu_r$ is the set of all $r${-th} roots of the unity.
\end{definition}

Apart from the quotient singularities, we can also define the following ones:

\begin{definition}\cite[Def. 2.8]{Kol13}
	Let $(X,D)$ be a log pair where $X$ is a normal variety of dimension $\geq 2$ and $D=\sum a_iD_i$ is a sum of distinct prime divisors, where $a_i$ are rational non-negative numbers, all $\leq 1$. Assume that $m(K_X+D)$ is Cartier for some $m>0$. We say that $(X,D)$ is 
	\[
	\left.
	\begin{array}{c}
		klt\\
		plt
	\end{array}
	\right\rbrace
	\text{ if } a(E,X,D)\;\;
	\left\lbrace
	\begin{array}{l}
		>-1 \text{ for every } E, \\
		>-1 \text{ for every exceptional } E.
	\end{array}
	\right.
	\]
	Here klt is short for \textit{`Kawamata log terminal'}, and plt for \textit{`purely log terminal'}.
\end{definition}

\begin{definition}
	We say that $X$ is \textit{log Fano} if there exists a divisor $D$ such that $-(K_X+D)$ is ample and $(X, D)$ is Kawamata log terminal.
\end{definition}

\begin{definition}\label{Def 2.6}
	Let $(X,\Delta)$ be a log pair and let $F$ be a divisor over $X$, i.e. $F$ is a divisor in a variety $Y$ such that $\pi:Y\to X$ is a birational morphism. When $F$ is a divisor on $X$ we write $\Delta=\Delta_1+aF$ where $F\nsubseteq \Supp(\Delta_1)$; otherwise let $\Delta_1=\Delta$.
	\begin{itemize}
		\item[(i)] $F$ is said to be \textit{primitive} over $X$ if there exists a projective birational morphism $\pi: Y \to X$ such that $Y$ is normal, $F$ is a prime divisor on $Y$ and $-F$ is a $\pi$-ample $\Q$-Cartier divisor. We call $\pi: Y\to X$ the \textit{associated prime blowup} (it is uniquely determined by $F$).
		\item[(ii)] $F$ is said to be of \textit{plt type} if it is primitive over $X$ and the pair $(Y,\Delta_Y+F)$ is plt in a neighbourhood of $F$, where $\pi: Y\to X$ is the associated prime blowup and $\Delta_Y$ is the strict transform of $\Delta_1$ on $Y$. When $(X,\Delta)$ is klt and $F$ is exceptional over $X$, $\pi$ is called a \textit{plt blowup} over $X$.
	\end{itemize}
\end{definition}

\subsection{Bounds for the stability threshold}\label{K sta def}

In this paper, we do not see the original definition of K-stability in terms of test configurations given in \cite{Tia97, Don02}. Instead, we focus on a simpler characterization of K-stability for Fano surfaces introduced in \cite{FO18}, which we use later in our computations. Recall that our setup is del Pezzo surfaces with at most klt singularities. However, {these results can be generalised to higher dimensions.}

\begin{definition}\label{basic def k-sta}
	Let $X$ be a normal variety of dimension $n$ such that $K_X$ is a $\Q$-Cartier divisor. Let $f:Y \to X$ be a birational morphism and take $E$ a prime divisor on $Y$. We define the \textit{expected vanishing order} as follows:
 \[
 S(-K_X;E)=\frac{1}{(-K_X)^n}\int_0^{\tau(E)} \vol (f^*(-K_X)-tE) \,dt
 \]
 where $\vol (f^*(-K_X)-tE)$ denotes the volume of the divisor $f^*(-K_X)-tE$ (see e.g. \cite[\S 2.2.C]{Laz04}) and
 \[
 \tau(E)=\sup\left\{ t\in \Q \vert \vol (f^*(-K_X)-tE)>0\right\}.
 \]
 $\tau(E)$ is called the \textit{pseudoeffective threshold}.
 \end{definition}

Now, we can introduce the notions to characterize the K-stability for Fano varieties.

\begin{definition}
	Let $X$ be a Fano variety with at most klt singularities and let $-K_X$ be its anticanonical divisor. The (adjoint) \textit{stability threshold} (or \textit{$\delta$-invariant}) of $X$ is defined as
	\begin{equation}\label{sthres}
		\delta(X)=\inf_{E/X} \frac{A_{X}(E)}{S(-K_X;E)}
	\end{equation}
	where the infimum runs over all divisors $E$ over $X$.
	
	We say that a divisor $E$ over $X$ {\textit{computes}} $\delta(X)$ if it achieves the infimum in \eqref{sthres}. 
\end{definition}
\begin{theorem}
    \cite{LXZ22,Li17,FO18,Fuj19b,CP21,BJ20}. If $X$ is a Fano variety, then $X$ is K-(semi)stable if and only if $\delta(X)> 1$ $(\geq 1)$.
\end{theorem}
There is also a local version of the stability threshold, which we focus on in this paper.

\begin{definition}\cite{AZ20}
	Let $X$ be a Fano variety with at most klt singularities and let $-K_X$ be its anticanonical divisor. Let $p$ be a closed point of $X$. We set
	\[
	\delta_p(X)=\inf_{E,p{\in} C_X(E)}\frac{A_{X}(E)}{S(-K_X;E)}
	\]
	where the infimum runs over all divisors $E$ over $X$ whose center contains $p$.
\end{definition}

Similar notions can be defined for $\N\times \N^r$-graded linear series with bounded support on $X$. Details can be found in  \cite[Section 3]{AZ20}, but here, we just recall the results we use.

\vskip 0.3cm

Let us fix a klt pair $(X,\Delta)$, some Cartier divisors $L_1,..., L_r$ on $X$ and an $\N^r$-graded linear series $V_{\overrightarrow{\bullet}}$ associated to the $L_i$'s such that $V_{\overrightarrow{\bullet}}$ contains an ample series and has bounded support. Let $F$ be a primitive divisor over $X$ with associated prime blowup $\pi: Y\to X$. Assume that $F$ is either Cartier on $Y$ or plt type and let $W_{\overrightarrow{\bullet}}$ be the refinement of $V_{\overrightarrow{\bullet}}$ by $F$.

\begin{theorem}\label{Theo 3.3} \cite[Theorem 3.3]{AZ20}
	Let $F$ be a primitive divisor over $X$ with associated prime blowup $\pi: Y\to X$. Let $Z\subset X$ be a subvariety, and let $Z_0$ be an irreducible component of $Z\cap C_X(F)$. Let $\Delta_Y$ be the strict transform of $\Delta$ on $Y$ (but remove the component $F$ as in Definition \ref{Def 2.6}) and let $\Delta_F=\Diff_F(\Delta_Y)$ be the different so that $(K_Y+\Delta_Y+F)\vert_F=K_F+\Delta_F$. Then we have
	\[
	\delta_Z(X,\Delta; V_{\overrightarrow{\bullet}})\geq \min \left\lbrace \frac{A_{X,\Delta}(F)}{S(V_{\overrightarrow{\bullet}};F)}, \inf_{Z'} \delta_{Z'}(F,\Delta_F;W_{\overrightarrow{\bullet}})\right\rbrace
	\]
	where $Z\subseteq C_X(F)$ and otherwise
	\[
	\delta_Z(X,\Delta; V_{\overrightarrow{\bullet}})\geq \inf_{Z'} \delta_{Z'}(F,\Delta_F;W_{\overrightarrow{\bullet}}),
	\]
	where the infimums run over all subvarieties $Z'\subseteq Y$ such that $\pi(Z')=Z_0$.
\end{theorem}


{The following Theorem is a direct consequence of  \cite[Theorem 3.2]{AZ20} and a simplification of \cite[Remark]{acc2021} in the surface case. It allows us to simplify the problem of finding a lower bound for the local stability threshold, using local stability thresholds of lower dimensional varieties.}


	

\begin{theorem}\label{Kento form}
Let $p$ be a point in {a del Pezzo surface $X$ with at most klt singularities}, let $\pi: Y\to X$ be a plt blowup of the point $p$, and let $E$ be the $\pi$-exceptional divisor. Then $(Y,E)$ has purely log terminal singularities, so that there exists an effective divisor $\Delta_E$ defined by $K_E+\Delta_E\sim_\Q(K_Y+E)\vert_E$. For every $t\in \left[ 0, \tau(-K_X;E)\right]$, where $\tau(-K_X;E)$ denotes the pseudo-effective threshold as in Definition \ref{basic def k-sta}, let us denote by $P(t)$ the positive part of the Zariski decomposition of the divisor $\pi^*(-K_X)-tE$, and let us denote by $N(t)$ its negative part. Let $W^E_{\cdot,\cdot}$ be {the refinement of $-K_X$ by $E$ (see \cite[\S 2.4]{AZ20} for the definition)}. Then
\[
\delta_p(X)\geq \min \left\{\frac{A_X(E)}{S(-K_X;E)},\; \min_{x\in E} \frac{1-\ord_x(\Delta_{E})}{S(W_{\cdot, \cdot}^{E}; x)}\right\}
 \]
 where for every $x\in E$ we have
 \[
 S(W_{\cdot, \cdot}^{E}; x)=\frac{2}{(-K_X)^2}\int_0^{\tau(-K_X;E)}h(t)\; dt
 \]
 where
 \[
 \begin{split}
      h(t) & =\left(P(t)\cdot E\right)\cdot \ord_x\left(N(t)\vert_E\right)+\int_0^\infty \vol_E\left( P(t)\vert_E-vx\right)\; dv\\
      &=\left(P(t)\cdot E\right)\times \left(N(t)\cdot E\right)_x+\int_0^{P(t)\cdot E} \left( P(t)\cdot E-v\right)\; dv\\
      & = \left(P(t)\cdot E\right)\times \left(N(t)\cdot E\right)_x+\frac{\left(P(t)\cdot E\right)^2}{2}
 \end{split}
 \]
\end{theorem}

{ This result gives a concrete recipe for what the flag of subvarieties of $X$ is, which is a Mori Dream space \cite{HK00}, and we apply it in Theorem \ref{> ineq}. Moreover, it} is one of the key points for the proof of the Main Theorem \ref{maintheo}. 

\section{Del Pezzo surfaces of degree 2}\label{dP sect}
As we mentioned in the introduction, the objects of study in this paper are the del Pezzo surfaces of degree 2. These surfaces are contained in weighted projective spaces {and are defined as follows:}

	

\begin{definition}
	A \textit{del Pezzo surface} or \textit{Fano surface} is a 2-dimensional Fano variety, i.e. a non-singular projective algebraic surface with ample anticanonical divisor class, $\omega_X^{-1}=-K_X$.

	The \textit{degree} $d$ of a del Pezzo surface X is given by the self-intersection of the anticanonical divisor, $(-K_X)^2=d$.
\end{definition}

\begin{remark}
{All of them but $\PP^1\times \PP^1$} can be represented as the blowup of $\PP^2$ at $r$ points in general position. In these cases, $d=9-r$.
\end{remark}

From now on, we focus on the case of del Pezzo surfaces of degree 2, which is represented by $X$. 

$X$ can be realized as the surface in weighted projective space $\PP(1,1,1,2)$ with homogeneous coordinates $x$, $y$, $z$, $w$, given by the equation
\begin{equation}\label{delPezzo2}
	w^2+wG_2(x,y,z)+G_4(x,y,z)=0,
\end{equation}
where $G_2(x,y,z)$ and $G_4(x,y,z)$ are weighted homogeneous polynomials of degrees 2 and 4, respectively. When we consider $X$ {over $\C$}, we may assume $G_2(x,y,z)=0$. Therefore, $X$ is given  by the equation:
\begin{equation}\label{dp2simply}
    w^2+G_4(x,y,z)=0.
\end{equation}
Notice that there exists $\rho:X\to \PP^2$ a double cover of $\PP^2$ ramified over a (smooth) quartic curve {$R$ \cite{Dol12}}, which is a canonical model of a curve of genus 3. 
\subsection{Cases covered in \cite{acc2021}}

In the proof of \cite[Lemma 2.15]{acc2021}, we can find a classification of the points in a smooth del Pezzo surface of degree 2, i.e. if {a point} $p\in X$, it satisfies one of the following:
\begin{itemize}
    \item[(1)] If we do an ordinary blowup of $p$ we get a smooth del Pezzo surface of degree 1;
    \item[(2)] $\rho(p)\in R$, and $C_p$ is an irreducible nodal curve;
    \item[(3)] $\rho(p)\in R$, and $C_p$ is an irreducible cuspidal curve;
    \item[(4)] $\rho(p)\in R$, and $C_p$ is a union of two (-1)-curves that meet transversally;
    \item[(5)] $\rho(p)\in R$, and $C_p$ is a union of two (-1)-curves that are tangent at $p$;
    \item[(6)] $\rho(p)\notin R$, and $p$ is contained in exactly one (-1)-curve;
    \item[(7)] $\rho(p)\notin R$, and $p$ is contained in two (-1)-curves;
    \item[(8)] $\rho(p)\notin R$, and $p$ is contained in three (-1)-curves;
    \item[(9)] the point $p$ is a generalized Eckardt point.
\end{itemize}
where $C_p$ is the unique curve in $\vert -K_X\vert$ that is singular at $p$. And in each case $\delta_p$ is computed to be:
\[
\delta_p(X)=\left\{\begin{array}{ll}
    {36}/{17} &  \text{ if (1)}\\
    2 &  \text{ if (2)}\\
   {15}/{8} &  \text{ if (3)}\\
    2 &  \text{ if (4)}\\
   {9}/{5} &  \text{ if (5)}\\
    {48}/{23} &  \text{ if (7)}\\
    {72}/{35} &  \text{ if (8)}\\
    2 &  \text{ if (9)}.
\end{array} \right.
\]
Notice that case (6) is left, when $p$ is contained in a unique $(-1)$-curve. However, the book \cite{acc2021} gives a bound for it, so we know that $\frac{40}{19}\geq \delta_p(X)\geq \frac{60}{31}$. Therefore, we focus on computing $\delta_p(X)$ for this case. 
It follows directly from equation \eqref{dp2simply}, that if $\rho(p)\notin R$ then the last coordinate of $p$ is non-zero.

\section{Weighted blowup of $X$}

In this section, we introduce some essential technical results for proving Theorem \ref{maintheo}. Let $\pi_{a,b}:X_{a,b}\to X$ be the $(a,b)$-weighted blowup of a del Pezzo surface of degree 2 at {a point $p_0\in X$. As in Theorem \ref{maintheo},} we are assuming {that there is a unique (-1)-curve $L$ passing through the point $p_0$}. We write their equations in local coordinates $u$ and $v$ such as $L=\{ u=0\}$, and $\wt(u)=a$ and $\wt(v)=b$ after the weighted blowup.
\begin{figure}[H]
    \centering
    \label{fig:(a,b) blowup}
    \caption{$(a,b)$-weighted blowup.}
    \includegraphics[width=1.1\linewidth]{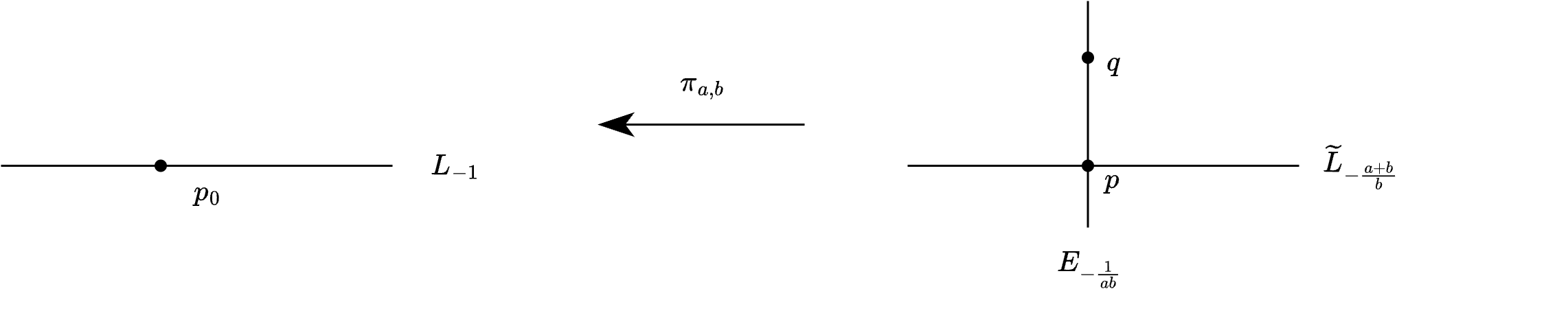}
\end{figure}
In the figure above, the subscripts represent the self-intersections, $E$ is the exceptional divisor of the weighted blowup, and {$\widetilde L$ is the strict transforms of $L$}.\\
In order to compute the Zariski Decomposition of the divisor $\pi_{a,b}^*(-K_X)-tE$, we write ${\pi_{a,b}^*}(-K_X)-tE$ as a non-negative combination of negative divisors {(See \cite[Theorem 2.3.19]{Laz04})}. {However, currently we do not have enough negative curves in the picture. To find a convenient expression, we search for another known algebraic variety that is birational to the weighted blowup. The idea is to find a linear equivalence in terms of the strict transform of $L$ in that space along with other negative curves, and bring it back to the weighted blowup $X_{a,b}$.}

\subsection{Notation}\label{not 2}
 As in Theorem $\ref{maintheo}$, let $p_0=(x_0,y_0,z_0,w_0)\in X$ be a closed point with $w_0\neq 0$. Let $X_{a,b}$ be the $(a,b)$-weighted blowup of $X$ at $p_0$. {Let $L$ be the unique (-1)-curve passing through $p_0$}. Let $E$ be the exceptional divisor of the weighted blowup. We assume $\frac{\sqrt{3}}{2}a\leq b\leq a$ and $\gcd(a,b)=1$. Therefore,  we can rewrite $a=b+\delta$ where $\delta\in \{1,...,b-1\}$. Similarly, we can rewrite $b=\delta\cdot i+\gamma_0$ and $\delta=\gamma_0\cdot j+\gamma_1$ where {$i,j\in \N$,} $\gamma_0\in\{0,1,...,\delta-1\}$ and $\gamma_1\in\{0,1,...,\gamma_0-1\}$. Generalizing this, let $\gamma_k=\gamma_{k+1}\cdot j_{k+1}+\gamma_{k+2}$ where {$j_{k+1}\in \N$} and $\gamma_{k+2}\in\{0,1,...,\gamma_{k+1}-1\}$. 

Since $\{\gamma_k\}$ is a decreasing sequence of natural numbers, there exists a $k_0\in \N$ where $\gamma_{k_0}=1$. Furthermore, we can choose $k_0$ such that for any other $k'\in \N$ where $\gamma_{k'}=1$, $k'\geq k_0$.

Let us denote $so_n=\sum_{m=1}^{n}j_{2m-1}$ and $se_n=\sum_{m=1}^{n}j_{2m}$ and for $n=0$, $so_0=se_0=0$.

\subsection{The resolution of $X_{a,b}$ followed by contractions to a weak degree 1 del Pezzo} \label{Finding D}

An $(a,b)$-weighted blowup of a smooth surface has at most two quotient singularities $p$ and $q$. If $a\neq 1$, $q$ is a $\frac{1}{a}(1,c_1)$ singularity and if $b\neq 1$, $p$ is a $\frac{1}{b}(1,d_1)$ singularity, where $c_1=-b+n_1 a$ and $d_1=-a+m_1 b$, with $n_1=\left\lceil \frac{b}{a}\right\rceil=1$ and $m_1=\left\lceil \frac{a}{b}\right\rceil=2$, respectively. \\
Similarly, we define the following:
\begin{itemize}
    \item $c_2=-a+n_2 c_1$, where $n_2=\left\lceil \frac{a}{c_1}\right\rceil$ and $c_k=-c_{k-2}+n_k c_{k-1}$ where $n_k=\left\lceil \frac{c_{k-2}}{c_{k-1}}\right\rceil$.
    \item $d_2=-b+m_2 d_1$, where $m_2=\left\lceil \frac{b}{d_1}\right\rceil$ and $d_k=-d_{k-2}+m_k d_{k-1}$ where $m_k=\left\lceil \frac{d_{k-2}}{d_{k-1}}\right\rceil$.
\end{itemize}

The resolution of these singularities is analogous to the Euclidean algorithm. The resolution of the singularity $p$ is achieved for some $i_0\in \N$ such that $d_{i_0}=1$ and similarly for $q$, the resolution will be complete when for some $j_0\in \N$ we get $c_{j_0}=1$.

\begin{remark}
Each $d_k$ can be represented as $d_k=\mu_k b+\lambda_k a$. Similarly, each $c_k$ can be represented as $c_k=\beta_k b+\alpha_k a$. 
Notice that the property $\bullet_k=-\bullet_{k-2}+m_k\cdot\bullet_{k-1}$ also holds for the coefficients $\mu_k$ and $\lambda_k$, and  $\bullet_k=-\bullet_{k-2}+n_k\cdot \bullet_{k-1}$ holds for $\alpha_k$ and $\beta_k$.
\end{remark}

From now on, we assume neither $a$ nor $b$ are equal to $1$. If this is not the case, it is enough to omit the singularity that corresponds to that weight. We present an algorithm to determine the resolution of $X_{a,b}$. It is achieved by repeatedly blowing up the singularities.

\paragraph{Resolution of $p$.} Let $\sigma_1:\hat{X}_{a,b}\to X_{a,b}$ be the {weighted} blowup of $X_{a,b}$  at {the quotient singular point $p$} {with suitable (natural) weight}. Let $E^{(1)}$ be the exceptional divisor and let $\hat{L}$ and $\hat{E}$ be the strict transforms of $\widetilde{L}$ and $E$ respectively. Since {$p$} is a $\frac{1}{b}(1,d_1)$ {quotient} singularity, $(E^{(1)})^2=-\frac{b}{d_1}$, $\sigma_1^*(\widetilde{L})\sim\hat{L}+\frac{d_1}{b}E^{(1)}$, and $\sigma_1^*(E)\sim\hat{E}+\frac{1}{b}E^{(1)}$. Therefore, we get the following self-intersections:
\[
\begin{aligned}
& (\hat{L})^2=\left(\sigma_1^*(\widetilde{L})-\frac{d_1}{b}E^{(1)}\right)^2=(\sigma_1^*(\widetilde{L}))^2+\left(\frac{d_1}{b}E^{(1)}\right)^2=-\frac{a+b}{b}-\frac{d_1}{b}=-(1+m_1)\\
& (\hat{E})^2=\left(\sigma_1^*({E})-\frac{1}{b}E^{(1)}\right)^2=(\sigma_1^*(E))^2+\left(\frac{1}{b}E^{(1)}\right)^2=-\frac{1}{ab}-\frac{1}{bd_1}=-\frac{m_1}{a d_1}= -\frac{\mu_1}{a d_1}.
\end{aligned}
\]

After this new blowup, if $d_1=1$, $\hat{X}_{a,b}$ will only have a singularity left and we continue with the resolution of $q$. On the other hand, if $d_1\neq 1$, $\hat{X}_{a,b}$ has a singularity of type $\frac{1}{d_1}(1,d_2)$ (as well as $q$) in the intersection of $E^{(1)}$ with $\hat{E}$ and we iterate the process by blowing it up. Denote this new blowup by $\sigma_2:\dot{X}_{a,b}\to\hat{X}_{a,b}$. Let $E^{(2)}$ the exceptional divisor of the $\sigma_2$ blowup, where $(E^{(2)})^2=-\frac{d_1}{d_2}$. Let $\dot{E}^{(1)}$ and $\dot{E}$ be the strict transforms of $E^{(1)}$ and $\hat{E}$, respectively. We get the following self-intersections:
\[
\begin{aligned}
& (\dot{E}^{(1)})^2 =\left(\sigma_2^*(E^{(1)})-\frac{d_2}{d_1}E^{(2)}\right)^2=(\sigma_2^*(E^{(1)}))^2+\left(\frac{d_2}{d_1}E^{(2)}\right)^2=-\frac{b}{d_1}-\frac{d_2}{d_1}=-m_2\\
& (\dot{E})^2 =\left(\sigma_1^*(\hat{E})-\frac{1}{d_1}E^{(2)}\right)^2=(\sigma_1^*(\hat{E}))^2+\left(\frac{1}{d_1}E^{(2)}\right)^2=-\frac{m_1}{a d_1}-\frac{1}{d_1 d_2}\\
&  \;\;\; \; \;\;\;=-\frac{(m_1 m_2-1) d_1}{a d_1 d_2}=-\frac{m_1 m_2-1}{a d_2}=-\frac{\mu_2}{a d_2}.
\end{aligned}
\]

Since $\gcd(a,b)=1$, by the Euclidean algorithm there exists an $i_0\in \N$ such that $d_{i_0}=1$. Therefore, the resolution of $p$ is achieved after $i_0$ steps. Let us denote $f_p:\Check{X}_{a,b}\to X_{a,b}$ the resolution of $p$ defined as $f_p=\sigma_{1}\circ\sigma_{2}\circ\cdots\circ\sigma_{i_0-1}\circ\sigma_{i_0}$. For each of these blowups {with suitable (natural) weight} of a quotient singular point, let us denote by $E^{(k)}$ the exceptional divisor of $\sigma_k$. Let $\Check{E}^{(k)}$, $\Check{E}$, $\Check{L}$, and $\Check{C}$ be the strict transforms of our divisors at the $i_0$-th blowup. Then, we have the following self-intersections:
\[
\begin{aligned}
& (\Check{E}^{(k)})^2=-m_{k+1}\text{ for } k=1,...,i_0-1, \;  \;\;\;(E^{(i_0)})^2=-d_{i_0-1}\\
&   (\Check{E})^2=-\frac{\mu_{i_0}}{a},\;\;\;\; (\Check{L})^2=-(1+m_1)=-3, \;\;\;\; (\Check{C})^2=3-\frac{b}{a}.
\end{aligned}
\]
 \begin{lemma}
  {$(\Check{E})^2=-\frac{\mu_{i_0}}{a}$.}
 \end{lemma}
\begin{proof}
 By induction, assume it is true for the $k$-th blowup, $(E)^2=-\frac{\mu_k}{ad_k}$. Let us check that it holds for the ($k+1$)-th blowup.
 \[
 \begin{aligned}
 (E)^2 & =-\frac{\mu_k}{a d_k}-\frac{1}{d_k d_{k+1}}=-\frac{\mu_k(-d_{k-1}+m_{k+1}d_k)+a}{a d_k d_{k+1}} = -\frac{a-\mu_k d_{k-1}+ \mu_k m_{k+1}d_k}{a d_k d_{k+1}}.
 \end{aligned}
 \]
It is enough to show that $a-\mu_k d_{k-1}=-\mu_{k-1} d_k$. Notice that we know $a=\mu_{k-1} d_{k-2}-\mu_{k-2} d_{k-1}$.
 \[
 \mu_k d_{k-1}-\mu_{k-1} d_k=(-\mu_{k-2}+m_k\mu_{k-1}) d_{k-1}-\mu_{k-1} (-d_{k-2}+m_kd_{k-1})=\mu_{k-1} d_{k-2}-\mu_{k-2} d_{k-1}=a
 \]
 Therefore, we get $(E)^2=-\frac{\mu_{k+1}}{a d_{k+1}}$. In particular, since $d_{i_0}=1$, $(\Check{E})^2=-\frac{\mu_{i_0}}{a}$.
 \end{proof}

The values of $i_0$, $d_{i_0-1}$, and $m_k$ can be specified using the notation of Subsection \ref{not 2}. We have two possibilities: 
\begin{itemize} 
\item If $k_0=2n_0$,
\[
\begin{aligned}
& i_0= i+so_{n_0}\;\text{ and }\; \left(\Check{E}^{\left(i+so_{n}+k\right)}\right)^{2}=-m_{\left(i+so_{n}+k+1\right)}=\left\{\begin{array}{ll}
-\left(j_{2 n}+2\right) & \text { if } k=0, \\
-2 & \text { Otherwise. }
\end{array}\right. \\
& (\Check{E}^{(i_0)})^2=-d_{i_0-1}=-(j_{2n_0}+1).
\end{aligned}
\]
\item If $k_0=2n_0+1$,
\[
\begin{aligned}
& i_0= i+so_{n_0+1}-1\;\text{ and }\; \left(\Check{E}^{\left(i+so_{n}+k\right)}\right)^{2}=-m_{\left(i+so_{n}+k+1\right)}=\left\{\begin{array}{ll}
-\left(j_{2 n}+2\right) & \text { if } k=0, \\
-2 & \text { Otherwise. }
\end{array}\right. \\
& (\Check{E}^{(i_0)})^2=-d_{i_0-1}=-2.
\end{aligned}
\]
\end{itemize}

\paragraph{Resolution of $q$.} Let $\tau_1:\overset{\cdot\cdot}{X}_{a,b}\to\Check{X}_{a,b}$ be the{weighted} blowup of $\Check{X}_{a,b}$ at {the quotient singular point} $q$ {with suitable (natural) weight}. Let $F^{(1)}$ be the exceptional divisor and $\overset{\cdot\cdot}{L}$, $\overset{\cdot\cdot}{E^{(k)}}$ and $\overset{\cdot\cdot}{E}$ be the strict transforms of $\Check{L}$, $\Check{E}^{(k)}$ and $\Check{E}$ respectively. Since $q$ is a $\frac{1}{a}(1,c_1)$ {quotient singularity}, $(F^{(1)})^2=-\frac{a}{c_1}$. Also, {since $q$ is in $\Check{E}$, its strict  transforms is not isomorphic to the pullback: $\tau_1^*(\Check{E})\sim\overset{\cdot\cdot}{E}+\frac{1}{a}F^{(1)}$.} Therefore:
\[
\begin{aligned}
& (\overset{\cdot\cdot}{E})^2=\left(\tau_1^*(\Check{E})-\frac{1}{a}F^{(1)}\right)^2=(\tau_1^*(\Check{E}))^2+\left(\frac{1}{a}F^{(1)}\right)^2=-\frac{\mu_{i_0 }}{a}-\frac{1}{ac_1}=-\frac{n_1\mu_{i_0}+\lambda_{i_0}}{c_1}\\
&\;\;\;\;\;\;\;=-\frac{\alpha_1\mu_{i_0}-\beta_1\lambda_{i_0}}{c_1}.
\end{aligned}
\]

After this blowup, if $c_1=1$, $\overset{\cdot\cdot}{X}_{a,b}$ is smooth and it is the resolution. On the other hand, if $c_1\neq 1$, $\overset{\cdot\cdot}{X}_{a,b}$ has a singularity of type $\frac{1}{c_1}(1,c_2)$ in the intersection of $F^{(1)}$ and $\overset{\cdot\cdot}{E}$, and we iterate the process by blowing it up, with the blow up denoted by $\tau_2:\overset{\cdots}{X}_{a,b}\to\overset{\cdot\cdot}{X}_{a,b}$. Let $F^{(2)}$ the exceptional divisor of the $\tau_2$ blowup, where $(F^{(2)})^2=-\frac{c_1}{c_2}$. Let $\overset{\cdots}{F^{(1)}}$ and $\overset{\cdots}{E}$ be the strict transforms of $F^{(1)}$ and $\overset{\cdot \cdot}{E}$, respectively. We get the following self-intersections:
\[
\begin{aligned}
& \left(\overset{\cdots}{F}^{(1)}\right)^2=\left(\tau_2^*(F^{(1)})-\frac{c_2}{c_1}F^{(2)}\right)^2=(\tau_2^*(F^{(1)}))^2+\left(\frac{c_2}{c_1}F^{(2)}\right)^2=-\frac{a}{c_1}-\frac{c_1}{c_2}=-n_2 \\
& (\overset{\cdots}{E})^2=\left(\tau_1^*(\Check{E})-\frac{1}{a}F^{(1)}\right)^2=(\tau_1^*(\Check{E}))^2+\left(\frac{1}{c_1}F^{(1)}\right)^2=-\frac{\mu_{i_0 }}{c_1}-\frac{1}{ac_1}=-\frac{\alpha_1\mu_{i_0}-\beta_1\lambda_{i_0}}{c_1}-\frac{1}{c_1 c_2}\\
&\;\;\;\;\;\;\;\;=-\frac{\alpha_2\mu_{i_0}-\beta_2\lambda_{i_0}}{c_2}
\end{aligned}
\]

Since $\gcd(a,b)=1$ by the Euclidean algorithm there exist $j_0\in \N$ such that $c_{j_0}=1$. Therefore, the resolution of $q$ is achieved after $j_0$ steps. Let us denote $f_q:\overline{X}_{a,b}\to \Check{X}_{a,b}$ the resolution of $q$ defined as $f_q=\tau_{1}\circ\tau_{2}\circ\cdots\circ\tau_{j_0-1}\circ\tau_{j_0}$. For each of these blowups {with suitable (natural) weight} of a quotient singular point, let us denote by $F^{(k)}$ the exceptional divisor of $\tau_k$. Let $\overline{F}^{(k)}$, $\overline{E}^{(k)}$, $\overline{E}$, and $\overline{L}$ be the strict transforms of our divisors after the $(i_0+j_0)$-th blowup. Then, we have the following self-intersections:
\[
 (\overline{F}^{(k)})^2=-n_{k+1}\text{ for } k=1,...,j_0-1, \;  \;\;\;(\overline{F}^{(j_0)})^2=-c_{j_0-1},\;\;\;\;   (\overline{E})^2=-1.
\]

\begin{remark}
{Notice that the self-intersections of $\overline{E}^{(k)}$ and $\overline{L}$ are the same as their images through $f_q$. }
\end{remark}
\begin{lemma}
   {$(\overline{E})^2=-1$.}
\end{lemma}
\begin{proof}
First we see that in the $(i_0+k)$-th blowup, $(E)^2=-\frac{\alpha_k\mu_{i_0}-\beta_k\lambda_{i_0}}{c_k}$. And then we see that $\alpha_{j_0}\mu_{i_0}-\beta_{j_0}\lambda_{i_0}=1$.
\\
As before, by induction on k, assume it is true for the $i_0+k$-th blowup. Then let us see that it also holds for the $k+1$-th.

\[
\begin{aligned}
(E)^2 & =-\frac{\alpha_k\mu_{i_0}-\beta_k\lambda_{i_0}}{c_k}-\frac{1}{c_k c_{k+1}}=-\frac{(\alpha_k\mu_{i_0}-\beta_k\lambda_{i_0})c_{k+1}+1}{c_k c_{k+1}}\\
 &=-\frac{(\alpha_k\mu_{i_0}-\beta_k\lambda_{i_0})(-c_{k-1}+n_{k+1}c_k)+1}{c_k c_{k+1}} \\
& =-\frac{1-(\alpha_k\mu_{i_0}-\beta_k\lambda_{i_0})c_{k-1}+(\alpha_k\mu_{i_0}-\beta_k\lambda_{i_0})n_{k+1}c_k}{c_k c_{k+1}}
\end{aligned}
\]
So it is enough to prove that $1-(\alpha_k\mu_{i_0}-\beta_k\lambda_{i_0})c_{k-1}=(\alpha_{k-1}\mu_{i_0}-\beta_{k-1}\lambda_{i_0})c_{k}$. 
\[
\begin{aligned}
& (\alpha_k\mu_{i_0}-\beta_k\lambda_{i_0})c_{k-1}+(\alpha_{k-1}\mu_{i_0}-\beta_{k-1}\lambda_{i_0})c_{k} \\
& =((-\alpha_{k-2}+n_k\alpha_{k-1})\mu_{i_0}-(-\beta_{k-2}+m_k\beta_{k-1})\lambda_{i_0})c_{k-1}+(\alpha_{k-1}\mu_{i_0}-\beta_{k-1}\lambda_{i_0})(-c_{k-2}+n_k c_{k-1}) \\
& =(\alpha_{k-1}\mu_{i_0}-\beta_{k-1}\lambda_{i_0})c_{k}=1
\end{aligned}
\]
Where for the first equation we use that $1=(\alpha_{k-1}\mu_{i_0}-\beta_{k-1}\lambda_{i_0})c_{k-2}+(\alpha_{k-2}\mu_{i_0}-\beta_{k-2}\lambda_{i_0})c_{k-1}$. Therefore, we get, $(E)^2=-\frac{\alpha_k\mu_{i_0}-\beta_k\lambda_{i_0}}{c_k}$. And in particular, since $c_{j_0}=1$ in the last blowup we get $(\overline{E})^2=-(\alpha_{j_0}\mu_{i_0}-\beta_{j_0}\lambda_{i_0})$. 
\\
Finally, we need to show that $\alpha_{j_0}\mu_{i_0}-\beta_{j_0}\lambda_{i_0}=1$. We know by definition that $\alpha_{j_0}a+\beta_{j_0}b=b\mu_{i_0}+a\lambda_{i_0}=1$. So, $a=n(\mu_{i_0}-\beta_{j_0})$, $b=n(\alpha_{j_0}-\lambda_{i_0})$. However, by the choice of $\mu_{i_0}$, $\lambda_{i_0}$, $\beta_{j_0}$ and $\alpha_{j_0}$, we conclude $n=1$. Then, since $\lambda_{i_0}=\alpha_{j_0}-b$, $\mu_{i_0}=\beta_{j_0}+a$,
\[
(\overline{E})^2=-(\alpha_{j_0}\mu_{i_0}-\beta_{j_0}\lambda_{i_0})=-1
\]
\end{proof}

The values of $j_0$, $c_{j_0-1}$ and $n_k$ can be specified using the notation in \S \ref{not 2}. We have two possibilities: 
\begin{itemize} 
\item If $k_0=2n_0$,
\[
\begin{aligned}
& j_0= j+se_{n_0}\;\text{ and }\; \left(\overline{F}^{\left(j+se_{n}+k\right)}\right)^{2}=-n_{\left(j+se_{n}+k+1\right)}=\left\{\begin{array}{ll}
-\left(j_{2 n+1}+2\right) & \text { if } k=1, \\
-2 & \text { Otherwise. } 
\end{array}\right. \\
& (\overline{F}^{(1)})^2=-n_2=-(i+2),\;\;\;\;(F^{(j_0)})^2=-c_{j_0-1}=-2.
\end{aligned}
\]
In the following picture, we can see all the divisors involved in the resolution of the singularities $p$ and $q$ when $k_0=2n_0$. Notice that the subscripts represent the self-intersections.
\begin{figure}[H]
    \centering
    \label{fig:Res of sin 2n0}
    \includegraphics[width=1\linewidth]{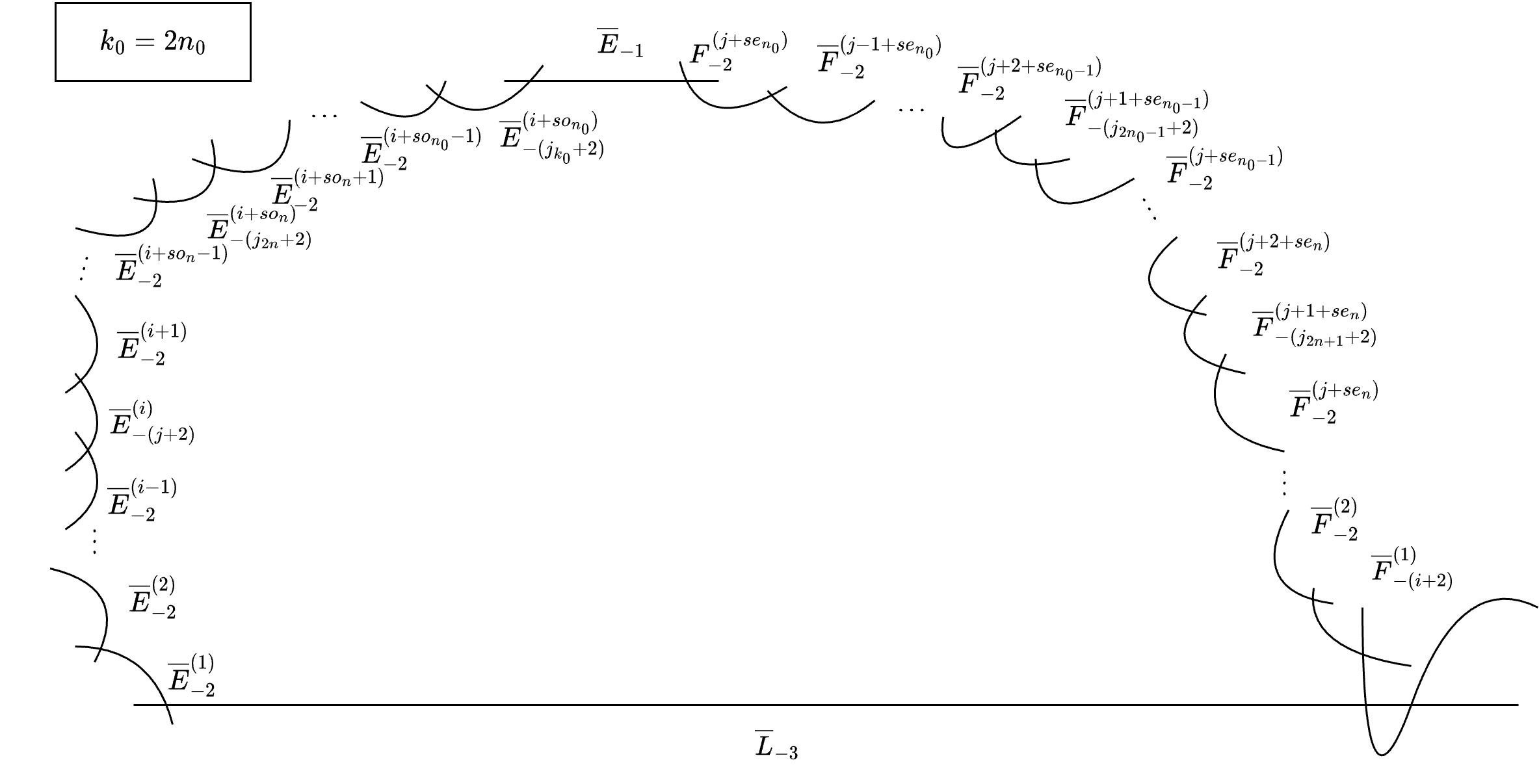}
\end{figure}
\item If $k_0=2n_0+1$,
\[
\begin{aligned}
& j_0= j+se_{n_0}+1\;\text{ and }\; \left(\overline{F}^{\left(j+se_{n}+k\right)}\right)^{2}=-n_{\left(j+se_{n}+k+1\right)}=\left\{\begin{array}{ll}
-\left(j_{2 n+1}+2\right) & \text { if } k=1, \\
-2 & \text { Otherwise. } 
\end{array}\right. \\
& (\overline{F}^{(1)})^2=-n_2=-(i+2),\;\;\;\;(F^{(j_0)})^2=-c_{j_0-1}=-(j_{2n_0+1}+1).
\end{aligned}
\]
In the following picture, we can see all the divisors involved in the resolution of the singularities $p$ and $q$ when $k_0=2n_0+1$. As before, the subscripts represent the self-intersections.
\begin{figure}[!ht]
    \centering
    \label{fig:Res of sin 2n0 1}
    \includegraphics[width=1\linewidth]{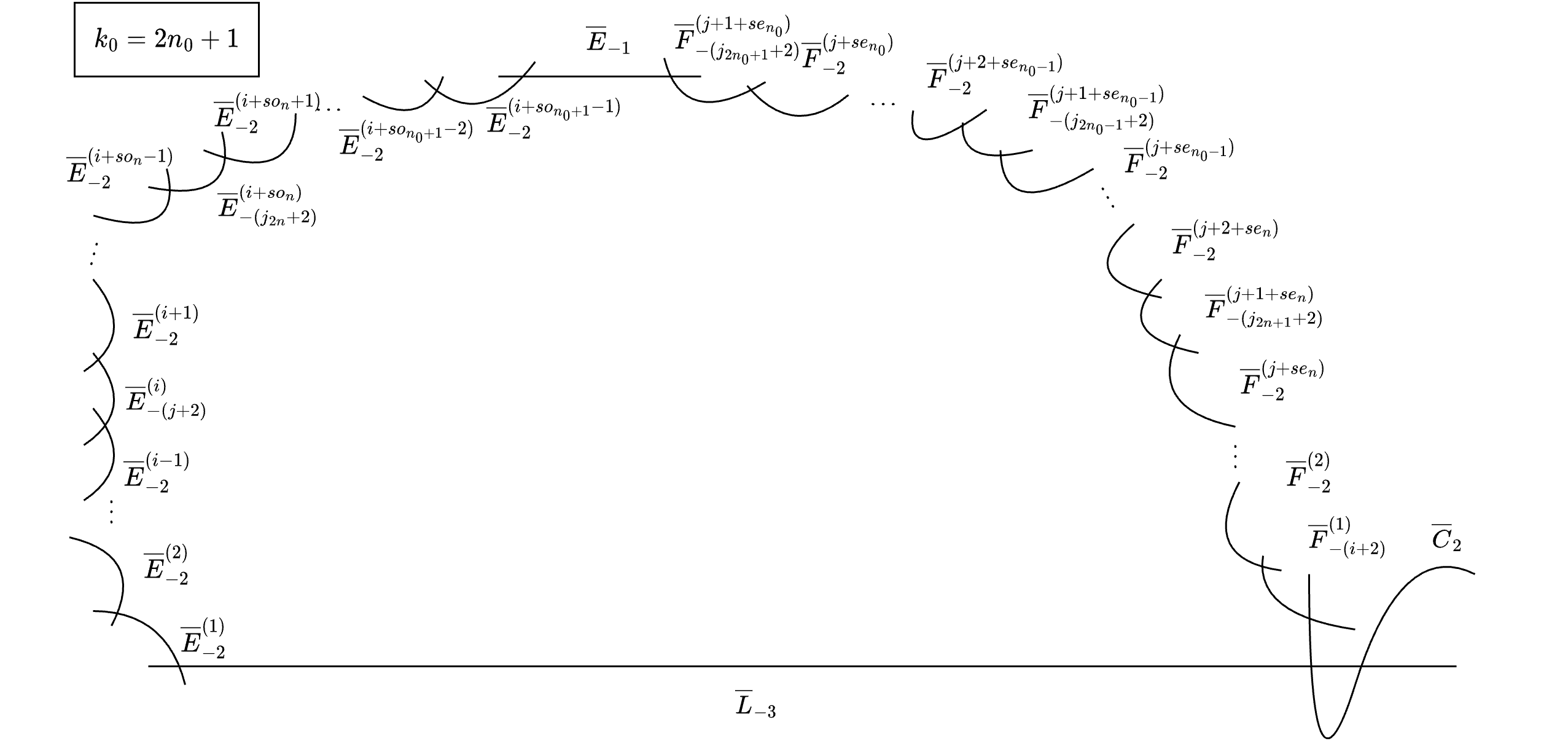}
\end{figure}
\end{itemize}
\newpage
Notice that in both cases, we can contract $(-1)$-curves until we get a weak del Pezzo degree one surface:

\begin{itemize}
    \item For $k=2n_0$:
    \begin{itemize}
        \item[(1)] Contract $\overline{E}$,
        \item[(2)] Contract $F^{(j+se_{n_0})}$,..., $\overline{F}^{(j+2+se_{n_0-1})}$ ($j_{2n_0}-1$ contractions).
        \item[(3)] Contract $\overline{E}^{(i+so_{n_0})}$,..., $\overline{E}^{(i+1+so_{n_0-1})}$ ($j_{2n_0-1}$ contractions).
        \item[(4)] Contract $\overline{F}^{(j+1+se_{n_0-1})}$,..., $\overline{F}^{(j+2+se_{n_0-1})}$ ($j_{2n_0-2}$ contractions).\\
        \vdots
        \item[($k_0+1$)] Contract $\overline{E}^{(i+j_1)}$,..., $\overline{E}^{(i+1)}$ ($j_{1}$ contractions).
        \item[($k_0+2$)] Contract $\overline{F}^{(j+1)}$,..., $\overline{F}^{(2)}$ ($j$ contractions).
        \item[($k_0+3 $)] Contract $\overline{E}^{(i)}$,..., $\overline{E}^{(1)}$ ($i$ contractions).
    \end{itemize}
    \item For $k=2n_0+1$:
    \begin{itemize}
        \item[(1)] Contract $\overline{E}$,
        \item[(2)] Contract $\overline{E}^{(i+so_{n_0+1}-1)}$,..., $\overline{E}^{(i+1+so_{n_0})}$ ($j_{2n_0+1}-1$ contractions).
        \item[(3)] Contract $F^{(j+se_{n_0}+1)}$,..., $\overline{F}^{(j+2+se_{n_0-1})}$ ($j_{2n_0}$ contractions).
        \item[(4)] Contract $\overline{E}^{(i+so_{n_0})}$,..., $\overline{E}^{(i+1+so_{n_0-1})}$ ($j_{2n_0-1}$ contractions).\\
        \vdots
        \item[($k_0+1$)] Contract $\overline{E}^{(i+j_1)}$,..., $\overline{E}^{(i+1)}$ ($j_{1}$ contractions).
        \item[($k_0+2$)] Contract $\overline{F}^{(j+1)}$,..., $\overline{F}^{(2)}$ ($j$ contractions).
        \item[($k_0+3 $)] Contract $\overline{E}^{(i)}$,..., $\overline{E}^{(1)}$ ($i$ contractions).
    \end{itemize}
\end{itemize}
In total, we contract $i_0+j_0=i+j+\sum_{k=1}^{k_0} j_k$ (-1)-curves in both cases. Let $g:\overline{X}_{a,b}\to \mathbf{X}$ the composition of all the contractions. 
The picture at the end is:

\begin{figure}[H]
    \centering
    \includegraphics[width=0.5\linewidth]{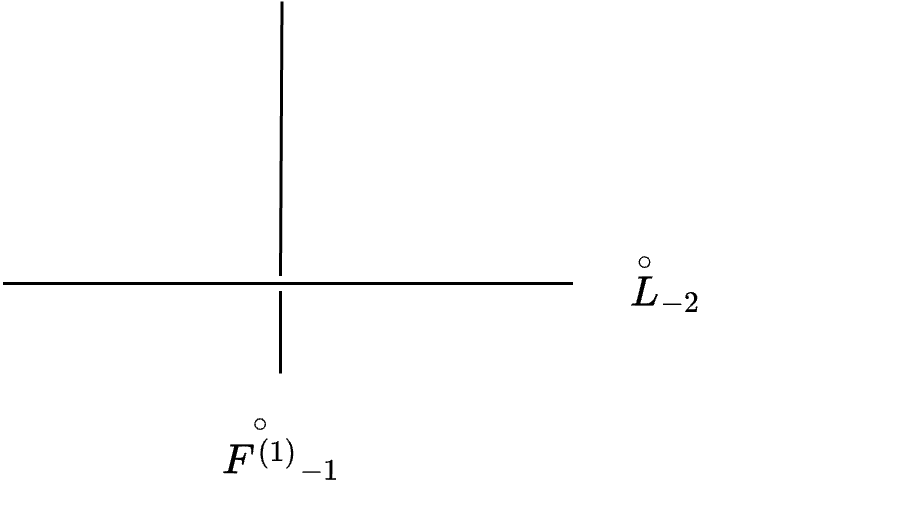}
\end{figure}

Where $\overset{\circ}{F^{(1)}}=g(\overline{F^{(1)}})$, and $\overset{\circ}{L}=g(\overline{L})$. So we have the following diagram:
{$$
    \xymatrix{
    \overline{X}_{a,b}\ar[d]_{f_q}\ar[ddr]^g& \\
    \Check{X}_{a,b}\ar[d]_{f_p}&  \\
    X_{ab}\ar[d]_{(\pi_{a,b})_{p_0}}& \mathbf{X}\ar[dl]^{(\pi_{1,1})_{p_0}}\\
    X&}
$$}

{Let us define $\pi_{1,1}= \pi_{a,b}\circ f_p\circ f_q\circ g^{-1}$}. We see now that this $\mathbf{X}$ is a weak del Pezzo surface of degree 1{, and $\pi_{1,1}$ is the ordinary blowup at the point $p
_0\in X$, where $\overset{\circ}{F^{(1)}}$ is the $\pi_{1,1}$-exceptional divisor}. 
\begin{proposition}
The surface $\mathbf{X}$ is a weak del Pezzo surface of degree 1.
\end{proposition}
\begin{proof}
 With each blowup we do for the resolution, we get an exceptional divisor, as we saw above, and these divisors are added with a nonpositive coefficient to the equivalence class of the anticanonical. For instance, with the weighted blowup we get  
\[
-K_{X_{a,b}}\sim \pi^*(-K_X)-(a+b-1)E
\]
Notice that the coefficients of the strict transforms of the existing divisors do not change and do not change in the successive ordinary blowups either. The same thing happens when we contract exceptional curves, the coefficients of the remaining exceptional curves do not change in the process.
Therefore, to prove that $(-K_\mathbf{X})^2=1$. We need to find the coefficient $e$ of $\overset{\circ}{F^{(1)}}$ in $-K_\mathbf{X}\sim {\pi_{1,1}^*(-K_X)}+eF^{(1)}$. \\
We got $F^{(1)}$, with the first blowup of the singular point $q\sim \frac{1}{a}(1,c_1)$, where $c_1=-b+a$. \\
Therefore, after the ordinary blowup at $q$, $\sigma: Y\to X_{a,b}$, (we are abusing notation here since this is not the same order of resolution we used above, but the coefficient of $F^{(1)}$ still is the same since it is independent of the order of the blowup), we get the following where $\overline{E}$ is the strict transform of $E$:
\[
\begin{array}{rl}
    -K_{Y}\sim & {\sigma^*\left(-K_{X_{a,b}}\right)+\frac{a-(c_1+1)}{a}F^{(1)}} \\
     \sim & {\sigma^*\left(\pi_{a,b}^*(-K_{X})-(a+b-1)E\right)+\frac{a-(c_1+1)}{a}F^{(1)}}\\
     \sim & {\sigma^*\left(\pi_{a,b}^*(-K_{X})\right)-(a+b-1)\bar{E}-\left(-\frac{a+b-1}{a}+\frac{a-(c_1+1)}{a}\right)F^{(1)}.}
\end{array}
\]
Moreover, notice that {$e=-1$}. 

So we have that, $(-K_\mathbf{X})^2={\left(\pi_{1,1}^*(-K_X)\right)^2+\left(\overset{\circ}{F^{(1)}}\right)^2=2-1=1}$. {It is clear that after contracting $\overset{\circ}{F^{(1)}}$ we will be back in the original degree 2 del Pezzo surface we started with. Therefore, we can say that $\pi_{1,1}: \mathbf{X}\to X$ is the ordinary blowup of $p_0\in X$.}
\end{proof}

\begin{remark}
In a weak del Pezzo surface of degree 1, we have a birational involution called the Bertini involution, $\iota$.
As in the proof of Lemma 2.15 in \cite{acc2021}, the linear system $\vert -2K_{\mathbf{X}}\vert$ gives a morphism $\mathbf{X}\to\PP(1,1,2)$ with the following Stein factorization:

$$\xymatrix{
\mathbf{X}\ar[r]^{\upsilon}& \widetilde{\mathbf{X}} \ar[d]^{\omega}\\
&\PP(1,1,2)
}$$
where $\upsilon$ is a contraction of all $(-2)$-curves in the surface $\mathbf{X}$ (if any), and $\omega$ is a double cover branched over the union of a sextic curve in $\PP(1,1,2)$ and the singular points of $\PP(1,1,2)$. Notice that $\widetilde{\mathbf{X}}$ is smooth if and only if $-K_{\mathbf{X}}$ is ample, and in that case, $\upsilon$ is an isomorphism. Furthermore, if the divisor $-K_{\mathbf{X}}$ is not ample, then the surface $\mathbf{X}$ contains at most four (-2)-curves. The double cover $\widetilde{\mathbf{X}}\to \PP(1,1,2)$ induces an involution $\iota\in\Aut(\mathbf{X})$ known as the Bertini involution.\\
For detailed equations of the Bertini involution, check \cite{Moo43}.

\end{remark}

Let $\overset{\circ}{D}=\iota(\overset{\circ}{F^{(1)}})$, and with this involution it is straightforward to check that
\[
\overset{\circ}{F^{(1)}}+\overset{\circ}{D}+\overset{\circ}{L}\sim-2K_{\mathbf{X}}.
\]
{Since $\pi_{1,1}$ contracts $\overset{\circ}{F^{(1)}}$ to $p_0$, we get that $-2K_X\sim L+D$, where $D=\pi_{1,1}(\overset{\circ}{D})$. }

\begin{lemma} Let $\overset{\circ}{D}$, $\overset{\circ}{F^{(1)}}$ and $\overset{\circ}{L}$ described as above. Then they do not intersect at the same point. 
\end{lemma}
\begin{proof}
In this proof, we have to take into account several things. 
    Notice that in $\widetilde{\mathbf{X}}$, $\upsilon(\overset{\circ}{F^{(1)}})\cdot\upsilon(\overset{\circ}{D})=3$ and both of them contain the singular point coming from the contraction of $\overset{\circ}{L}$. 
    Then, after the blowup of the singular point at $\widetilde{\mathbf{X}}$, we recover $\overset{\circ}{L}$ again, and since $\overset{\circ}{F^{(1)}}+\overset{\circ}{D}+\overset{\circ}{L}\sim-2K_{\mathbf{X}}$  the intersections change as follows:
    \[
    \overset{\circ}{F^{(1)}}\cdot \overset{\circ}{D}=2,\;\;\; \overset{\circ}{F^{(1)}}\cdot \overset{\circ}{L}=\overset{\circ}{L}\cdot \overset{\circ}{D}=1.
    \]
    This implies that $\overset{\circ}{F^{(1)}}\cap \overset{\circ}{D}\cap\overset{\circ}{L}=\emptyset$.
\end{proof}
{This is the image representing the curves involved after applying $\pi_{1,1}$:}
\begin{figure}[H]
    \centering
    \label{fig:(1,1) blowup}
    \caption{Ordinary blowup.}
    \includegraphics[width=1.1\linewidth]{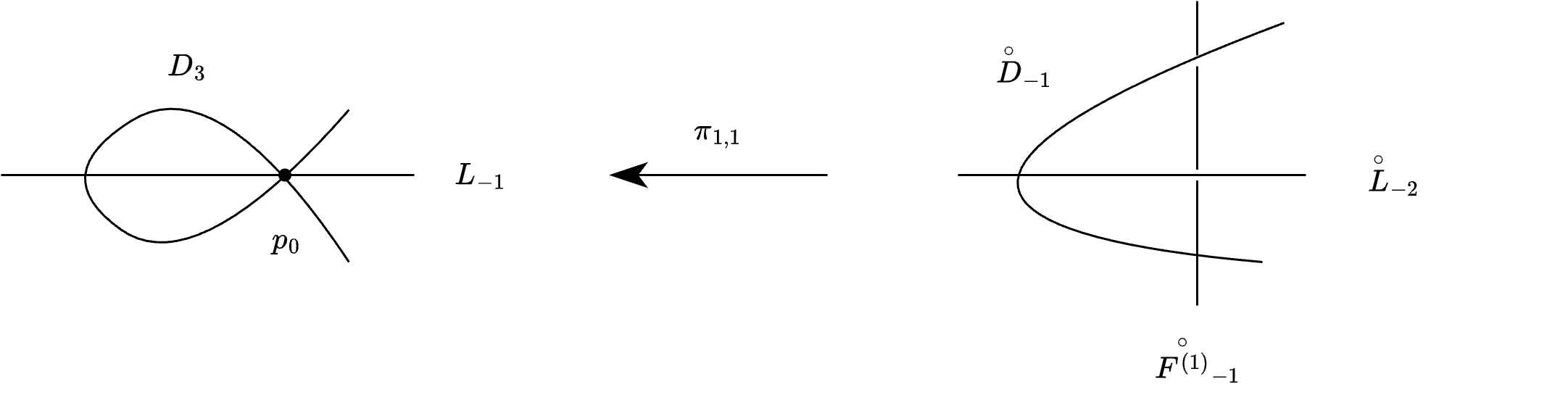}
\end{figure}
{In this figure, we can also see how the intersections change.}

As a consequence, we have the following proposition:

\begin{proposition}\label{D} {Assume that $3a<4b$ and set $\widetilde D=(f_p\circ f_q)(g^*(\overset{\circ}{D}))$} {on} $X_{a,b}$. {Then, $\pi_{a,b}^*(D)=\widetilde D+2bE$} and $(\widetilde D)^2=3-\frac{4b}{a}<0$.
 \end{proposition}
\begin{proof}
The idea of this proof is to take the equivalence {$\pi_{1,1}^*(D)\sim \overset{\circ}{D}+2\overset{\circ}{F^{(1)}}$ in the weak degree 1 del Pezzo surface} $\mathbf{X}$, and follow the inverse process {of blow ups and contractions} to get an equivalence for {$\pi_{a,b}^*(D)$} in $X_{a,b}$ in terms of $\widetilde D$ and $E$. 
Let $e_k$ and $f_k$ be the coefficient of $E^{(k)}$ and $F^{(k)}$, respectively, in the equivalence class of {$\pi_{a,b}^*(D)$}. Define $e$ to be the final coefficient of $E$. Notice that if $k=2n_0$, $e=e_{i+so_{n_0}}+f_{j+se_{n_0}}$, otherwise $e=e_{i+so_{n_0+1}-1}+f_{j+se_{n_0}+1}$. In addition, we know for the sequence of blowups that for $n\geq 0$
\[
\begin{array}{rlr}
   e_{i+so_{n}+k}=  & e_{i+so_{n}}+k\cdot f_{j+se_{n}+1} & \text{for } k=1,...,j_{2n+1},\\
   f_{j+se_{n}+k+1}=  & k\cdot e_{i+so_{n+1}} + f_{j+se_{n}+1} & \text{for } k=1,...,j_{2n+2}.
\end{array}
\]
Now let us rewrite $b$ in terms of $i$, $j$, and $j_k$ using notation as in \S \ref{not 2}. 
\[
{
\begin{array}{rl}
    2b= & 2\delta \cdot i+2\gamma_0=\delta \cdot e_i+f_1\cdot \gamma_0= \gamma_0(e_i\cdot j+f_1)+\gamma_1\cdot e_i 
 =\gamma_0\cdot f_{j+1}+\gamma_1\cdot e_i =\\
   =  & \gamma_1(j_1 \cdot f_{j+1}+e_i)+\gamma_2 \cdot f_{j+1}= \gamma_1\cdot e_{i+so_1}+\gamma_2 \cdot f_{j+1}=...=\\
   = &\gamma_{2n_0-2}f_{j+se_{n_0-1}+1}+\gamma_{2n_0-1}e_{i+se_{n_0-1}}=\gamma_{2n_0-1}e_{i+so_{n_0}}+\gamma_{2n_0}f_{j+se_{n_0-1}+1}
\end{array}
}
\]
Here we have two different cases. 
\begin{itemize}
\item If $k_0=2n_0$, notice that,
$\gamma_{2n_0-1}=j_{2n_0}$ and $\gamma_{2n_0}=1$. Therefore,
\[
2b=j_{2n_0}e_{i+so_{n_0}}+f_{j+se_{n_0-1}+1}=f_{j+se_{n_0}}+e_{i+so_{n_0}}=e.
\]
\item If $k_0=2n_0+1$, notice that $\gamma_{2n_0}=j_{2n_0+1}$ and $\gamma_{2n_0+1}=1$. Therefore,
\[
2b=j_{2n_0+1}f_{j+se_{n_0}+1}+e_{i+so_{n_0}}=e_{i+so_{n_0+1}-1}+f_{j+se_{n_0}+1}=e.
\]
\end{itemize}
In both cases, as desired, we obtain $e=2b$. Therefore, after all the contractions, we get {$\pi_{a,b}^*(D)\sim \widetilde D+2b\cdot E$}. To compute the self-intersection of $\widetilde D$:
\[
{
(\widetilde D)^2=(\pi_{a,b}^*(D)-2bE)^2=\left(\pi_{a,b}^*(D)\right)^2-4b\left(\pi_{a,b}^*(D)\right)\cdot E+4b^2E^2=3-\frac{4b}{a}<0 \;\;\Leftrightarrow \;\; 3a<4b.}
\]
\end{proof}
{
\begin{remark}
    Following a similar proof and taking into account that $a=b+\delta$, it is easy to check that $\pi_{a,b}^*(L)=\widetilde L+aE$, where $\widetilde L$ is the strict transform of $L$ after the weighted blowup.
\end{remark}
\begin{corollary}
    Let $\pi_{a,b}: X_{a,b}\to X$ be a $(a,b)$-weighted blowup of the point $p_0\in X$, as described in Theorem \ref{maintheo}. Let $L$ be the unique (-1)-curve passing through the point $p_0$ and $D$ the curve described above. Then, 
    \[
\pi_{a,b}^*(-K_X)\sim \frac{1}{2}\left(\widetilde L+ \widetilde D +(a+2b)E\right).
    \]
\end{corollary}
}

\section{Proof of the Main Theorem}\label{pmaintheo}

In the previous section, we prepared all the ingredients to prove the Main Theorem. Now, we divide this proof into two results:

\begin{theorem}\label{< ineq}
Let $X\subset \PP (1,1,1,2)$ be a smooth del Pezzo surface of degree 2 and let $p_0=(x_0,y_0,z_0,w_0)\in X$ be a closed point with $w_0\neq 0$. Assume {that there is a unique (-1)-curve $L$ passing through the point $p_0$}. Then $\delta_{p_0}(X)\leq\frac{6}{71} (11 + 8 \sqrt{3})$.
\end{theorem}

\begin{proof}
For each $a,b>0$, let $\nu_{a,b}$ be the quasi-monomial valuation over ${p_0}\in X$ defined by $\nu_{a,b}(u)=a$ and $\nu_{a,b}(v)=b$. We need to identify the minimizer of $\frac{A_X(\nu_{a,b})}{S_X(\nu_{a,b})}$. In order to do this, choose coprime integers $a,b>0$ and let $\pi: Y=X_{a,b}\to X$ be the weighted blow up at $p_0$ with $\wt(u)=a$ and $\wt(v)=b$. Let $E$ be the exceptional divisor and let $\widetilde{L}$, and {$\widetilde{D}$} be the strict transforms of $L$ and {$D$}, respectively, {where $D$ is described above}. Assume that $\frac{\sqrt{3}}{2}a<b<a$.
	
\vskip 0.3cm	
By definition, $\pi^*(L)=\widetilde{L}+m_L E$ and {$\pi^*(D)=\widetilde{D}+m_C E$} where in our case  $m_L=\mult_{p_0}L=a$ and {$m_D=\mult_{p_0}D=2b$ as we proved in Theorem \ref{D}}. It is known that $E^2=-\frac{1}{ab}$, so we get: $(\widetilde{L})^2=-\frac{a+b}{b}$, {$(\widetilde{D})^2=3-\frac{4b}{a}$, $(\widetilde{L})\cdot E=\frac{1}{b}$, $(\widetilde{D})\cdot E=\frac{2}{a}$, and $\widetilde{D}\cdot \widetilde{L}=1$}.

As we saw in the previous section, the stable base locus of 
\[
\pi^*(-K_X)-tE\sim  \frac{1}{2}\left(\widetilde{L}+\widetilde{D}\right)+\left(\frac{a+2b}{2}-t\right)E
\]
is contained in $\widetilde{D}\cap \widetilde{L}$ for all $0\leq t\leq \frac{a+2b}{2}$ and we have the following Zariski decomposition:

\[
    N(\pi^*(-K_X)-tE)=\left\lbrace
	\begin{array}{ll}
		0 & 0\leq t\leq b,\\
		\frac{t-b}{a+b}\widetilde{L} & b<t\leq \frac{a(3a+4b)}{3a+2b},\\
  \frac{3t(2b-a)-4b^2}{4b^2-3a^2} \widetilde{L}+\frac{t(3a+2b)-a(3a+4b)}{4b^2-3a^2} \widetilde{D} & \frac{a(3a+4b)}{3a+2b}<t\leq \frac{a+2b}{2},
	\end{array}
	\right.
\]
and
\[
	P(\pi^*(-K_X)-tE)=\left\lbrace
	\begin{array}{ll}
		\frac{1}{2}\left(\widetilde{L}+\widetilde{D}\right)+\left(\frac{a+2b-2t}{2}\right)E & 0\leq t\leq b\\
	\frac{a+3b-2t}{2(a+b)}\widetilde{L}+\frac{1}{2}\widetilde{D}+\left(\frac{a+2b-2t}{2}\right)E & b<t\leq \frac{a(3a+4b)}{3a+2b}\\
  \frac{3t(2b-a)(a+2b-2t)}{2(4b^2-3a^2)} \widetilde{L}+\frac{(3a+2b)(a+2b-2t)}{2(4b^2-3a^2)} \widetilde{D} +\left(\frac{a+2b-2t}{2}\right)E & \frac{a(3a+4b)}{3a+2b}<t\leq \frac{a+2b}{2}
	\end{array}
	\right.
\]
Then, by \cite[Theorem B, Example 4.7]{BFJ09} and \cite[Corollary C]{LM09}, we know that 
\[
	\vol_{Y\vert E}(-\pi^*K_X-tE)=P_\sigma(\pi^*(-K_X)-tE)\cdot E=-\frac{1}{2}\cdot \frac{d}{dt}\vol (-\pi^*K_X-tE)
\]
Thus, we get
\[
	\vol_{Y\vert E}(\pi^*(-K_X)-tE)=\left\lbrace
	\begin{array}{ll}
		\frac{t}{ab} & 0\leq t\leq b\\
		\frac{a+t}{a(a+b)} & b< t \leq \frac{a(3a+4b)}{3a+2b}\\
		\frac{6(a+2b-2t)}{4b^2-3a^2} & \frac{a(3a+4b)}{3a+2b}<t\leq \frac{a+2b}{2}
	\end{array}
	\right.
\]
and
\[
	S(-K_X;E)=\frac{2}{(-K_X)^2}\int_0^\infty t\cdot \vol_{Y\vert E}(\pi^*(-K_X)-tE)dt=\frac{15a^2+34ab+8b^2}{12(3a+2b)}.
\]
Since $A_X(E)=a+b$, note that $\frac{A_X(\nu_{a,b})}{S_X(\nu_{a,b})}$ only depends on the ratio $\mu=\frac{a}{b}$, thus by continuity \cite[Proposition 2.4]{BLX22} we have
\[
	\frac{A_X(\nu_{a,b})}{S_X(\nu_{a,b})}=\frac{12(3\mu+2)(1+\mu)}{15\mu^2+34\mu+8}
\]
for all $a,b\in \R_+$. It achieves its minimum for $\frac{2}{\sqrt{3}}\geq\mu\geq\frac{2}{3}$ with the value $\lambda_0=\frac{6}{71} (11 + 8 \sqrt{3})$ at $\mu_0:=\frac{2}{\sqrt{3}}$. In particular, we have $\delta_{p_0}(X)\leq \lambda$. 
\end{proof}

It remains to show $\delta_{p_0}(X)\geq \lambda$. We use the same method as in Lemma A.6 of \cite{AZ20}.

\begin{theorem}\label{> ineq}
Let $X\subset \PP (1,1,1,2)$ be a smooth del Pezzo surface of degree 2 and let $p_0=(x_0,y_0,z_0,w_0)\in X$ be a closed point with $w_0\neq 0$. Assume {that there is a unique (-1)-curve $L$ passing through the point $p_0$}. Then $\delta_{p_0}(X)\geq\frac{6}{71} (11 + 8 \sqrt{3})$.
\end{theorem}
\begin{proof}
Choose a sequence of coprime integers $a_m,b_m>0$ ($m=1,2,...$) such that $\mu_m:=\frac{a_m}{b_m}\to \mu_0$ ($m\to\infty$), where $\frac{2}{3}<\mu_m<\frac{2}{\sqrt{3}}$. Let $\pi_m: Y_m=Y_{a_m,b_m}\to X$ be the corresponding weighted blow up and let $E_m$ be the exceptional divisor. Let $P_1^{(m)}=\widetilde{L}\cap E_m$, {$\lbrace P_2^{(m)}, P_3^{(m)}\rbrace=\widetilde{D}\cap E_m$, where $P_1^{(m)}$ and $P_2^{(m)}$ are the singular points $p$ and $q$, respectively}. Let $W_{\overrightarrow{\cdot,\cdot}}^{E_m}$ be the refinement by $E_m$ of the complete linear series associated to $-K_X$. \\

Let $\Delta_m=\Diff_{E_m}(0)=(1-\frac{1}{b_m})P_1^{(m)}+(1-\frac{1}{a_m})P_2^{(m)}$, $\lambda_m=\frac{A_X(E_m)}{S(-K_X;E_m)}$. Now using the formula in Corollary \ref{Kento form} we know that: 
 \[
\delta_p(X)\geq \min \left\{\frac{A_X(E_m)}{S(-K_X;E_m)},\; \min_{x\in E_m} \frac{1-\ord_x(\Delta_{E_m})}{S(W_{\cdot, \cdot}^{E_m}; x)}\right\}
 \]
We already computed the minimum of $\frac{A_X(E_m)}{S(-K_X;E_m)}$ in Theorem \ref{< ineq}. Therefore, to finish the proof, we need the compute the other minimum. Notice that in $E_m$ we have two singularities, $P_1^{(m)}$ and $P_2^{(m)}$, so we have four types of points $x\in E_m$. Here, we denote by $P(t)$ the positive part of the Zariski decomposition of Theorem \ref{< ineq}.
\begin{itemize}
    \item If $x=P_1^{(m)}$, 
    \[
    \frac{1-\ord_x(\Delta_{E_m})}{S(W_{\cdot, \cdot}^{E_m}; x)}=\frac{1-\left(1-\frac{1}{b_m}\right)}{S(W_{\cdot, \cdot}^{E_m}; P_1^{(m)})}=\frac{12(2b_m+3a_m)^2}{45a_m^2+60a_mb_m+44b_m^2},
    \]
    where
    \[
    \begin{split}
    S(W_{\cdot,\cdot}^{E_m};P_1^{(m)}) & =\int_0^{\frac{a+2b}{2}}\left( \left( P(t)\cdot E_m\right)\cdot \left( N(t)\cdot E_m\right)_{P_1^{(m)}}+\frac{\left( P(t)\cdot E_m\right)^2 }{2}  \right) \; d t\\
    & =\frac{45a_m^2+60a_mb_m+44b_m^2}{12b_m(2b_m+3a_m)^2}
    \end{split}
    \]
    Now since $\mu_m\to \mu$ as $m\to \infty$, let us take the limit
    \begin{equation}\label{P1}
    \lim_{m\to \infty} \frac{1}{b \cdot S(W_{\cdot, \cdot}^{E_m}; P_1^{(m)})}=\frac{6}{47} \left(11 + 3 \sqrt{3}\right)
    \end{equation}
    \item If $x=P_2^{(m)}$,
    \[
    \frac{1-\ord_x(\Delta_{E_m})}{S(W_{\cdot, \cdot}^{E_m}; x)}=\frac{1-\left(1-\frac{1}{a_m}\right)}{S(W_{\cdot, \cdot}^{E_m}; P_2^{(m)})}=\frac{3(2b_m+3a_m)^2}{18a_m^2+12a_mb_m+4b_m^2},
    \]
    where
    \[
    \begin{split}
    S(W_{\cdot,\cdot}^{E_m};P_2^{(m)})=\frac{2(9a_m^2+6a_mb_m+2b_m^2)}{3a_m(2b_m+3a_m)^2}
    \end{split}
    \]
    Now since $\mu_m\to \mu$ as $m\to \infty$, let us take the limit
    \begin{equation}\label{P2}
    \lim_{m\to \infty} \frac{1}{a\cdot S(W_{\cdot, \cdot}^{E_m}; P_2^{(m)})}=\frac{6}{37} \left(8 + 3 \sqrt{3}\right)
    \end{equation}
   \item {If $x=P_3^{(m)}$,
    \[
    \frac{1-\ord_x(\Delta_{E_m})}{S(W_{\cdot, \cdot}^{E_m}; x)}=\frac{1}{S(W_{\cdot, \cdot}^{E_m}; P_3^{(m)})}=\frac{3a_m(2b_m+3a_m)^2}{18a_m^2+12a_mb_m+4b_m^2},
    \]
    where
    \[
    \begin{split}
    S(W_{\cdot,\cdot}^{E_m};P_3^{(m)})=\frac{2(9a_m^2+6a_mb_m+2b_m^2)}{3a_m(2b_m+3a_m)^2}
    \end{split}
    \]
    Now since $\mu_m\to \mu$ as $m\to \infty$, let us take the limit
    \begin{equation}\label{P3}
    \lim_{m\to \infty} \frac{1}{ S(W_{\cdot, \cdot}^{E_m}; P_3^{(m)})}=\frac{12}{37} \left(8 + 3 \sqrt{3}\right)
    \end{equation}}
    \item If $x\neq P_1^{(m)},\,P_2^{(m)},\, P_3^{(m)}$,
    \[
    \frac{1-\ord_x(\Delta_{E_m})}{S(W_{\cdot, \cdot}^{E_m}; x)}=\frac{1}{S(W_{\cdot, \cdot}^{E_m}; x)}=\frac{18a_m^2 + 12a_mb_m}{15a_m - 2b_m},
    \]
    where
    \[
    \begin{split}
    S(W_{\cdot,\cdot}^{E_m};x)=\frac{15a_m - 2b_m}{18a_m^2 + 12a_mb_m}
    \end{split}
    \]
    Now since $\mu_m\to \mu$ as $m\to \infty$, let us take the limit
    \begin{equation}\label{genx}
    \lim_{m\to \infty} \frac{1}{ S(W_{\cdot, \cdot}^{E_m}; x)}=\frac{12}{37} \left(8 + 3 \sqrt{3}\right)
    \end{equation}
    If we put it all together, we get that
     \[
\delta_p(X)\geq \min \left\{\lambda,\eqref{P1}, \eqref{P2},\eqref{P3},\eqref{genx}\right\}=\lambda=\frac{6}{71} (11 + 8 \sqrt{3})
 \]
\end{itemize}
\end{proof}
This concludes the proof of the Main Theorem \ref{maintheo}.

{\footnotesize
\bibliography{mybibtex}}
\Addresses

\end{document}